\documentclass[
10pt]{amsart}
\usepackage{amssymb}
\allowdisplaybreaks
\usepackage[normalem]{ulem}
\usepackage{xcolor}

\newtheorem{theorem}{Theorem}[section]
\newtheorem{lemma}[theorem]{Lemma}

\newtheorem{cor}[theorem]{Corollary}

\newtheorem{proposition}[theorem]{Proposition}

\theoremstyle{definition}
\newtheorem{definition}[theorem]{Definition}

\theoremstyle{remark}
\newtheorem{remark}[theorem]{Remark}

\numberwithin{equation}{section}

\newcommand{\cC}{\mathcal{C}}
\newcommand{\cG}{\mathcal{G}}
\newcommand{\cI}{\mathcal{I}}
\newcommand{\cK}{\mathcal{K}}
\newcommand{\MM}{\mathcal{M}}
\newcommand{\OO}{\mathcal{O}}
\newcommand{\cS}{\mathcal{S}}
\newcommand{\cW}{\mathcal{W}}

\newcommand{\hW}{\widehat{\cW}}

\newcommand{\htop}{h_{\scriptsize{\mbox{top}}}}
\newcommand{\vf}{\varphi}
\newcommand{\ve}{\varepsilon}

\begin{document}

\title{Measure of maximal entropy for finite horizon Sinai billiard  flows}

\author{Viviane Baladi\textsuperscript{(1),(2)}, J\'er\^ome Carrand\textsuperscript{(1),(4)}, and Mark Demers\textsuperscript{(3)}}
 
\email{baladi@lpsm.paris}
\email{jerome.carrand@sns.it}
\email{mdemers@fairfield.edu}
\address{(1) Sorbonne Universit\'e and Universit\'e Paris Cit\'e, CNRS,   Laboratoire de Probabilit\'es, Statistique et Mod\'elisation,
	F-75005 Paris, France}
\address{(2) Institute for Theoretical Studies, ETH, 8092 Z\"urich,
	Switzerland}
\address{(3) Department of Mathematics, Fairfield University, Fairfield CT 06824, USA}
\address{(4) Centro di Ricerca Matematica Ennio De Giorgi, SNS, 
56100 Pisa,  Italy (current address)}

\thanks{Part of this work was done during a workshop at ICMS, Edinburgh in June 2022.
	The research of VB  and JC is supported
	by the European Research Council (ERC) under the European Union's Horizon 2020 research and innovation programme (grant agreement No 787304).  MD is partially supported by National Science Foundation grant DMS 2055070.
We thank F. P\`ene for pointing out several typos and the referee for suggestions which helped us to greatly improve
the exposition}.
\begin{abstract}
	Using recent work of Carrand 
	on equilibrium states for the billiard map, and adapting techniques from
	 Baladi and Demers, 
	we construct the unique measure of maximal entropy (MME) for
	 two-dimensional finite horizon Sinai (dispersive) billiard flows $\Phi^1$
(and show it is Bernoulli),  assuming the bound
$\htop(\Phi^1) \tau_{\min} > s_0 \log 2$,
where  $s_0\in (0,1)$ quantifies the recurrence to singularities. This bound holds in
many examples (it is expected to hold generically).
\end{abstract}

\date{\today}

\maketitle
\section{Introduction and Main Result}

\subsection{Background}
Let $\Phi^t$ be a continuous flow on a compact manifold.
The topological entropy of the flow, $\htop(\Phi^1)$,
 is the supremum, over  ergodic  probability measures $\nu$
invariant under the (continuous) time-one map $\Phi^1$ of the Kolmogorov entropy $h_\nu(\Phi ^1)$. If a measure realising the supremum exists, it  is called a measure of 
maximal entropy (MME) for the flow.

\smallskip
For geodesic flows, the study of  the MME  has a rich history.  In the case of  strictly negative curvature, the flow is Anosov, i.e. 
smooth and uniformly hyperbolic, and the pioneering works of
 Bowen \cite{Bow} and  Margulis 
 \cite{Mar0, Mar} half a century ago established existence, uniqueness, and mixing of
the MME, leading to  remarkable consequences, in particular
on the structure (counting and equidistribution) of periodic orbits. 
For more general continuous flows, it became apparent \cite{Bo0, Bo1, BW} that (flow) expansivity implies
existence of the MME, and combined  \cite{Fr} with the
(Bowen) specification 
property,  also gives uniqueness.

Starting with the groundbreaking work of Knieper
\cite{Kn}, most developments in the past 25 years have concerned smooth geodesic flows 
 for which the hyperbolicity or compactness
assumption are relaxed.
In recent years,
Climenhaga and Thompson \cite{CT} have revisited
the Bowen specification approach, which has allowed them
to obtain several striking \cite{CKW, BCFT} results.

\smallskip
Sinai billiard flows, our object of study,  are natural dynamical systems which are uniformly hyperbolic,
but not differentiable
(we  refer to \cite{CM} for a  full-fledged introduction to mathematical billiards):
A Sinai   billiard table $Q$  on the two-torus $\mathbb{T}^2$
is a set
$
Q=\mathbb{T}^2 \setminus \cup_{i} \OO_i
$,
for  finitely many 
pairwise disjoint  convex closed domains $\OO_i$  with $C^3$ boundaries 
having strictly positive curvature $\cK$. 
The
billiard flow  $\Phi^t$, $t \in \mathbb{R}$,   is 
the motion of a point particle traveling in $Q$ at unit speed 
and
undergoing specular reflections\footnote{At a tangential  collision, the
reflection does not change the direction of the particle.} at the boundary of
the scatterers $\OO_i$.  
The associated billiard map $T:M\to M$,  on the compact metric set
$M = \partial Q \times [-\frac{\pi}{2}, \frac{\pi}{2} ]$, is
the first collision map on the boundary of $Q$. Grazing
collisions cause discontinuities in the map $T$, but
 the flow is continuous (after identification of the incoming
and outgoing angles). The map is expansive \cite{BD1},
but this property is not automatically\footnote{See \cite{BW} for a definition
of expansiveness for the flow. See \cite[Ex. 1.6]{Bo0} for
a weaker sufficient condition for existence.} inherited by the flow, 
since neither the map nor the return time is continuous.
In particular, it is not obvious that
 the flow 
 satisfies a  condition (such as  asymptotic $h$-expansiveness \cite{Mi})
 sufficient for the upper-semi continuity of
 the Kolmogorov entropy  (see \cite[App. A--B]{Ca}), and there  does not appear to exist 
 an unconditional proof of the existence --- let alone uniqueness ---
of a MME for  the billiard flow.

\smallskip

The purpose of the present paper is to furnish mild conditions guaranteeing
existence, uniqueness, and mixing (in fact, the Bernoulli property) of the MME for Sinai
billiards.  This can be viewed
as a first step towards the much harder open problem of establishing equidistribution
results for Sinai billiards.

\smallskip
Our results are stated precisely in \S1.2, after furnishing the necessary notations. In particular, 
Corollary~\ref{cormain} of  Theorem~\ref{main} guarantees existence, uniqueness and Bernoullicity  of the MME for all
finite horizon Sinai
billiard flows $\Phi^t$ such that 
\begin{equation}\label{flows}
\htop(\Phi^1) \tau_{\min} > s_0 \log 2\, ,
\end{equation} where $\tau_{\min}$
is the minimum time between collisions, and $s_0\in (0,1)$ quantifies the recurrence rate
to singularities. The sufficient condition for the existence,  uniqueness and Bernoullicity of the MME for billiard maps
obtained in \cite{BD1} is 
\begin{equation}\label{maps}
h_* > s_0 \log2\, ,
\end{equation} where $h_*>0$ is a combinatorial  definition of the topological
entropy of the map (see \eqref{defh}). 
We show below (see the last claim of Lemma~\ref{mild}) that \eqref{flows} implies
\eqref{maps}.
Section~2.4 of \cite{BD1} describes two  billiard classes 
(periodic Lorentz gas with disks of radius $1$ centered in a triangular lattice,
and periodic Lorentz gas with two scatterers of different radii on the unit square
lattice) where  \eqref{maps}  can be
checked  for many parameters. In Remark 5.6 of \cite{Ca}, the author checks 
\eqref{flows}  for an open subset of these parameters.
No example is known where \eqref{maps} or \eqref{flows} can be shown \emph{not} to hold.

\smallskip

Our proof is based on previous work of Carrand \cite{Ca} (Chapter 3 of his thesis \cite{Ca0}, itself relying
on \cite{BD1}) and on \cite{BD2}.
These three papers use the\footnote{To our knowledge, the Climenhaga--Thompson specification approach has not been implemented
yet for Sinai billiards.} technique of transfer operators
acting on anisotropic spaces, which was first introduced to billiards 
by Demers--Zhang \cite{DZ1}, and recently applied to 
construct the measure of maximal entropy of the billiard map 
\cite{BD1}.

\subsection{Results}

To state our main results, Theorem~\ref{main} and\footnote{The condition \eqref{sparse} there is discussed in Lemma~\ref{mild}.} Corollary~\ref{cormain}, we introduce some basic notation.
For $x \in M$, let $\tau(x)$ denote the flow time (return time) from $x$ to $T(x)$,
let $\cK_{\min}=\inf \cK>0$,
and set 
$$\tau_{\min}=\inf \tau>0\, , \,\, \tau_{\max}=\sup \tau\, ,\,\, 
\Lambda = 1 + 2 \tau_{\min} \cK_{\min}\, .
$$
Throughout, we assume finite horizon, that is: there are no trajectories making only tangential collisions. Finite horizon implies $\tau_{\max}<\infty$.

\medskip

Set
$$
P(-t\tau)=\sup_{\mu : T\mbox{\tiny{-invariant ergodic probability measure}}}
\biggl \{h_\mu(T)-t \int \tau d \mu \biggr \} \, , \,\,t\ge 0\, .
$$
The real number $P(-t \tau)$ is called the pressure of the potential $-t\tau$,
and a probability measure $\mu_t$ realising $P(-t \tau)$ is called an equilibrium measure
for $-t\tau$.  For simplicity, we just\footnote{In \cite{BD2} we studied
$P(-t\log J^u T))=\sup_\mu \{h_\mu(T)-t \int \log J^u T \, d \mu  \} $, for $J^uT$ the unstable Jacobian of $T$.
There is no risk of confusion since we only consider $P(-t\tau)$ in the present paper.}
 write $P(t)$ instead of $P(-t\tau)$.

Viewing $\Phi$ as the suspension of $T$ under
$\tau$, 
Abramov's formula says that
any ergodic  probability measure $\nu$ 
invariant under   the time-one map $\Phi^1$  satisfies 
\begin{equation}
\label{Abr0}
\nu=\frac{\mu}{\int  \tau d \mu} \otimes Leb\, ,
\end{equation} 
where
$\mu$ is an ergodic $T$-invariant probability measure, and, in addition,
\begin{equation}
\label{Abr}
h_\nu(\Phi ^1)
=\frac{h_\mu(T)}{\int  \tau d \mu} \, .
\end{equation}
\smallskip

In the  coordinates $x = (r,\vf)$, where $r$ is arclength along
$\partial \OO_i$ and $\vf$ is the  post-collision angle with the normal to 
$\partial \OO_i$,
let $\cS_0 = \{ (r, \vf) \in M: \vf = \pm \frac{\pi}{2} \}$ denote the set of tangential collisions on $M$.
Then  for any $n\in \mathbb Z_*$, the set
$
\cS_n = \cup_{i=0}^{-n} T^i\cS_0$
is the singularity set of $T^n$.
Following \cite{BD1}, define $\MM_0^n$  to be the
set of maximal connected components of
$M \setminus \cS_n$ for $n\ge 1$, and set
\begin{align}\label{defh}
h_*=\lim_{n\to \infty} \frac 1 n \log \# \MM_0^n 
\end{align}
(existence of the limit is easy \cite{BD1}).
Then, for fixed $\vf<\pi/2$ close to $\pi/2$ and large $n \in \mathbb{N}$, define
$s_0(\vf,n) \in (0,1]$ to be the smallest number such that 
any orbit of
length equal to $n$ has at most $s_0n$ collisions whose
angles with the
normal are larger than $\vf$ in absolute value.
 If there exist $\varphi$ and $n$ such that   $s_0 = s_0(\vf,n)$ satisfies
 \begin{equation}\label{sparse0}
 h_*> s_0\log 2 \, ,
 \end{equation}
then \cite{BD1} proves that $P(0)=h_*$, and
there is a unique equilibrium measure $\mu_*=\mu_0$ for $t=0$, which
is  the unique MME of $T$.
As already mentioned, there are many billiards \cite[\S 2.4]{BD1} satisfying \eqref{sparse0}, and in fact
we do not know any billiard which violates it.
Moreover, Demers and Korepanov showed \cite{DK} that 
	a conjecture of B\'alint and T\'oth  \cite{balint toth}, if true,  implies that,
	  for generic finite horizon configurations of scatterers, one can choose $\vf$ and $n$ to make $s_0$ arbitrarily small.

\medskip

Using Abramov's formula, Carrand showed the following:
\begin{proposition}
[{\cite[Lemma 2.5 and its proof, Cor. 2.6]{Ca}}]\label{Car2}
The function $t\mapsto P(t)$ is continuous and strictly
decreasing on $(-\infty, \infty)$, with $-\lim_{t\to \pm \infty}P(t)=\pm\infty$.
The real number $t=\htop(\Phi^1)>0$ is the unique $t$ such that $P(t)=0$.
In addition, the set of equilibrium measures of $T$ for $-\htop(\Phi^1)\tau$
is in bijection with the set of MMEs  of the flow via \eqref{Abr0}.
\end{proposition}

Denote $\Sigma_n \tau:=\sum_{k=0}^{n-1} \tau\circ T^k$ (to avoid confusion with $\cS_n$ and the  notation $S^\delta_n$ below). We next state Carrand's main results   (see also Proposition~\ref{Car3} below).
 
 \begin{theorem}[{\cite[Theorem 2.1, Theorem 1.2]{Ca}}]\label{Car}
(a)~The following\footnote{By \cite{BD1} we always have $P_*(0)=h_*\ge P(0)$.} limits exist: 
$$
 P_*(t)=\lim_{n\to \infty} \frac 1 n \log Q_n(t)\, ,
 \mbox{ with } Q_n(t)=  \sum_{A \in \MM_0^n} |e^{-t\Sigma_n \tau}|_{C^0(A)}\, ,\,
 \forall t\ge 0\, .
 $$
 Moreover, $P_*(t)> P_*(s)\ge P(s)$ for all $0\le t < s$, and\,\footnote{The fact that
 $P_*(t)$ is strictly decreasing is immediate, see \eqref{forconv}.
 Convexity follows from the H\"older inequality as in \cite[Prop 2.6]{BD2}.}  $t\mapsto P_*(t)$ is convex. 
 
\smallskip
\noindent (b)~If $t\ge 0$ is such that
 \begin{equation}\label{sparse}
 P_*(t)+ t \tau_{\min} > s_0 \log 2 \, ,
 \end{equation} 
 and
 \begin{equation}\label{hassle}
 \log \Lambda > t (\tau_{\max}-\tau_{\min})\, ,
 \end{equation}
 then there is  a unique  equilibrium measure $\mu_t$ for  $-t\tau$.
 This measure charges all open sets, is Bernoulli,
 and $P_*(t)=P(t)$. Finally, $\mu_t$ is
 $T$-adapted,\footnote{\label{foot:adapt}To establish \eqref{adapt}, Carrand shows that
 the $\mu_t$ measure of the $\epsilon$-neighbourhood of $\cS_{\pm 1}$
is bounded by $C_t |\log \epsilon|^{ - \gamma}$ for  some $\gamma >1$ and $C_t<\infty$.} that is
 \begin{equation}\label{adapt}
 \int |\log d(x, \cS_{\pm 1})| \, d\mu_t<\infty\, . 
 \end{equation}
 \end{theorem}

The work of Lima and Matheus \cite{LM}  shows the usefulness of the $T$-adapted property.

\smallskip

In view of Proposition~\ref{Car2} and Theorem~\ref{Car}, to establish existence and
uniqueness of the MME of the finite horizon flow $\Phi$, it \emph{suffices}
to check \eqref{sparse} and \eqref{hassle} for $t=\htop(\Phi^1)>0$. We next discuss these conditions. The first one is very mild:

\begin{lemma}\label{mild}
The bound \eqref{sparse} holds  at $t=\htop(\Phi^1)$ as soon as
\begin{equation}
\label{only}
\htop(\Phi^1) \tau_{\min}> s_0 \log 2 \, .
\end{equation}
The bound \eqref{only} holds as soon as 
\begin{equation}
\label{only'}
	h_* \frac{\tau_{\min}}{\tau_{\max}} > s_0 \log 2\, .
\end{equation}
	
\noindent If \eqref{sparse} holds for  some $t'\ge 0$ then it holds for all $t\in [0,t']$. 
In particular, if \eqref{sparse} holds  at $t=\htop(\Phi^1)$ then
\eqref{sparse0} holds since $P_*(0)=h_*$.
	\end{lemma}

It is not hard to find  \cite[Remark 5.6]{Ca}  billiards satisfying  \eqref{only}. 
The idea there is to compare a computable lower bound for the left-hand-side of \eqref{only} with an upper bound 
for the right-hand-side. In the examples from \cite[\S 2.4]{BD1}, this comparison is sufficient to check that \eqref{only} holds, 
as long as $\tau_{\min}$ is large enough.

\begin{proof}
The first claim follows from  Proposition~\ref{Car2} and  the bound $P_*(t)\ge P(t)$ for all $t\ge 0$.
The second claim holds because  \eqref{Abr} implies $
	\htop(\Phi^1)\ge \frac{h_*}{\int \tau d\mu_*} \ge \frac{h_*}{\tau_{\max}}$.
Finally,  the first claim of Lemma~\ref{lem:t_*} below implies that
$t\mapsto P_*(t)+t\tau_{\min}$ is nonincreasing.
\end{proof}

Obviously, for any finite horizon billiard, there exists $\tilde t>0$ such that  \eqref{hassle}
	holds for all $t\in[0, \tilde t]$.
However, we do\footnote{Note that \eqref{Abr} implies 
$\htop(\Phi^1) (\tau_{\max}-\tau_{\min}) \le h_* (\tau_{\max}/\tau_{\min}-1)$.} not know any billiard such that \eqref{hassle}  can be verified
for $t=\htop(\Phi^1)$, that is,
$
\log \Lambda > \htop(\Phi^1) (\tau_{\max}-\tau_{\min})$. 
Fortunately, it turns out that 
 \eqref{hassle} is not \emph{necessary:} 
Assuming only finite horizon and \eqref{sparse}  at $t=\htop(\Phi^1)$,
we will extend the conclusion of Theorem~\ref{Car} to $t=\htop(\Phi^1)$ 
by adapting the bootstrapping argument in \cite[Lemma 3.10]{BD2} 
(used there to cross the value $x=1$ at which the pressure for  $-x\log J^u T$ vanishes). This is our main result:

\begin{theorem}\label{main}
Let $T$ be  a finite horizon Sinai billiard map such that  \eqref{sparse} holds  at $\htop(\Phi^1)$.
Then 
for all $t\in[0, \htop(\Phi^1)]$,
we have $P_*(t)=P(t)$, and there exists a unique $T$-invariant probability measure $\mu_t$
realising $P(t)$. This measure charges all nonempty open sets, is Bernoulli and $T$-adapted.
\end{theorem}

Our proof furnishes 
$t_\infty> \htop(\Phi^1) $ such that the key Small Singular Pressure 
properties  \eqref{keyb}, \eqref{keyb3}, and \eqref{keyb2} hold for all $t\in[0, t_\infty]$. 
Note that if \eqref{sparse} holds
at some $t_2\in (\htop(\Phi^1), t_\infty]$,  the conclusion of Theorem~\ref{main} holds for
all $t\in [0,t_2]$.

Theorem~\ref{Car} and Proposition~\ref{Car2} of Carrand,   combined with  Theorem~\ref{main}
and the proof of \cite[Props. 7.1 and 7.2]{Ca} for Bernoullicity of the flow, give:

\begin{cor}\label{cormain}
Let $T$ be  a finite horizon Sinai billiard map such that  \eqref{sparse} holds  at $t=\htop(\Phi^1)$.
Then 
$$
\nu_*:=\frac{\mu_{\htop(\Phi^1)}}{\int \tau \, d \mu_{\htop(\Phi^1)}} \otimes Leb\, 
$$
is the unique measure of maximal entropy of the billiard flow.
This measure is Bernoulli, it charges all nonempty open sets, and  it  is flow adapted, that
is\footnote{Note that \eqref{flad} implies that $\log \| D\Phi_t \|$ is integrable for each $t \in [-\tau_{\min}, \tau_{\min}]$ so that, by subadditivity, it is integrable
for each $t \in \mathbb{R}$.}
\begin{equation}\label{flad}
\int_\Omega |\log d_\Omega(x, \cS_0^\pm) | \, d\nu_{*} <\infty\, ,\quad
\Omega =  Q \times \mathbb{S}^1  
\, ,
\end{equation}
where  
$d_\Omega$ is the Euclidean metric, 
$
\cS_0^- = \{ \Phi_{-s}(z) 
: z \in \cS_0 \, ,\,  s \le \tau(T^{-1}z) \}$,
and $ \cS_0^+ = \{ \Phi_s(z) 
: z \in \cS_0 \, ,\,   s \le \tau(z) \}$.
\end{cor}

\smallskip

Contrary to \cite{BD2}, homogeneity layers are not used for our potentials $-t\tau$. They are not needed
because $\tau$ is piecewise H\"older and thus  $e^\tau$  satisfies piecewise bounded
distortion. The results of Carrand \cite{Ca} that we build upon
are based on bounds for transfer operators acting on Banach spaces of distributions
defined with the logarithmic modulus of
continuity of \cite{BD1}. We could not find a Banach norm  giving a spectral gap (there is no analogue
of \cite[Lemmas 3.3 and 3.4]{BD2} for $\varsigma\ne 0$,
see  \cite[Lemma 3.1]{Ca} for $\gamma\ne 0$ where $(\log |W|/\log|W_i|)^\gamma$ replaces $(|W_i|/|W|)^\varsigma$). We thus do not have exponential mixing
for $(T,\mu_{\htop(\Phi^1)})$. (Even if we had, it would not immediately imply exponential 
mixing for $(\Phi^1,\nu_*)$.)

\medskip

The paper is organised as follows:
Section~\ref{GL} is devoted to recalling notation from \cite{BD1} and to two basic lemmas on 
cone stable curves iterated by the billiard map. 

Section~\ref{boot}
contains key ingredients from \cite{Ca} as well as the crucial
new definition \eqref{bart}, as we explain next: To show Theorem~\ref{Car},
Carrand introduced a key technical condition of Small Singular Pressure (SSP).
The pressure $P_*(t)$ is a thermodynamic limit corresponding to sums (for the weight $\exp(-t\tau)$
arising from Abramov's formula)
over stable curves iterated (in the
past), and cut by billiard singularities. As usual for hyperbolic systems with singularities, 
for fixed $t>0$, we must
see that the contraction coming from the weight $\exp(-t\tau)\le \exp (-t\tau_{\min})$ beats the growth 
due to summing over bits fragmented by the singularities.
(This is necessary to get good bounds on the iterated transfer operators associated
to the map $T$ and the weight $\exp(-t\tau)$. These bounds are needed to construct  maximal eigenvectors
for this operator and its dual on suitable Banach spaces of distributions.)
Condition SSP for a parameter $t>0$ essentially says that there exists a scale $\delta_t>0$
such that, at all large times, the contribution of those thermodynamic sums 
 which correspond
to curves which have become  shorter than $\delta_t/3$ is at most a  controlled fraction of the sum over all
curves.
In \S\ref{SSP}, we first recall the SSP conditions \eqref{keyb},  \eqref{keyb3}, and \eqref{keyb2} from \cite{Ca},
and we then state Carrand's  conditional Theorem~\ref{Car3}. This theorem says that, if SSP holds at $t$, then there is a unique equilibrium measure for the potential $-t\tau$, and it thus reduces Theorem~\ref{main} to showing
SSP for some $t\ge \htop(\Phi^1)$.
We  set up the bootstrap mechanism by introducing in \eqref{bart} the supremum $t_\infty>0$ 
of parameters satisfying SSP  (this is the new idea). 
The first key lemma, Lemma~\ref{lem:low t>},  inspired by 
\cite[Lemma 3.10]{BD2} 
exploits the H\"older inequality to estimate weighted thermodynamic sums for $t$
by using the pressure $P_*(u)$ and its one-sided derivative $P'_{*,-}(u)$, for $0<u\le t<t_\infty$.
It is stated and proved in \S\ref{KL}.

The actual bootstrapping argument
is carried out in \S\ref{44}. 
Lemma~\ref{lem:t_*}   embodies our version of ``pressure gap''  (inspired by  \cite[Definition 3.9]{BD2}): This lemma constructs a ``{pivot}'' $t_*<t_\infty$ 
and its associated parameter  $s_*(t_*)>t_\infty$. (The pivot is chosen in such a way that
the first key lemma  can be exploited at $u=t_*$.)
Lemma~\ref{lem:short small t>},
the second key lemma (inspired by
\cite[Lemma 3.11]{BD2}),
says that, if $P_*(t_*)\ge 0$,
then SSP holds in the interval $[t_*,s_*(t_*))$.
(The proof uses the first key lemma, taking advantage of the choice of the pivot.)
Finally,  Theorem~\ref{main} is proved 
in \S\ref{proofmain}: We assume for a contradiction that $t_\infty<\htop(\Phi^1)$.
Since $t_*<t_\infty$, this implies, by results from
\cite{Ca} recalled in Proposition~\ref{Car2} and Theorem~\ref{Car}(a), that the pressure   $P_*(t_*)$ is 
nonnegative.
The second key lemma can thus be applied and  gives the desired contradiction
since $s_*(t_*)>t_\infty$.


\section{Notations. $n$-Step Expansion. Growth Lemma}\label{GL}

We recall here some facts about hyperbolicity and complexity of finite horizon Sinai billiards.
There exist continuous families of stable and unstable cones, $\cC^s$ and $\cC^u$, which can be taken
constant in $M$, and a constant $C_1 \in (0,1)$ such that,
\begin{equation}
\label{eq:hyp}
\| DT^n(x) v \| \ge C_1 \Lambda^n \| v \| \, ,\; \forall v \in \cC^u \, , \quad
\| DT^{-n}(x) v \| \ge C_1 \Lambda^n \| v \| \, , \; \forall v \in \cC^s\, ,
\end{equation}
where, as before, $\Lambda = 1 + 2 \tau_{\min} \cK_{\min}$ is the minimum hyperbolicity constant.

A fundamental fact about this class of billiards is the linear bound on the growth in complexity due to
Bunimovich \cite[Lemma~5.2]{Ch},
\begin{equation}
\label{eq:complex}
\begin{split}
& \mbox{There exists  $K \ge 1$ such that for all $n \ge 0$, the number of curves in $\cS_{\pm n}$}  \\
& \mbox{that intersect
at a single point is at most $Kn$.}
\end{split}
\end{equation}

 The parameter $\gamma >1$ defining the Banach space norms in \cite{Ca} is chosen so that 
$h_* > s_0 \gamma \log 2$, which is possible due to \eqref{sparse0}.
Next, choosing $m$ so large that,
\[
\tfrac 1m \log (Km+1) < h_* - s_0 \gamma \log 2 \, ,
\]
we take $\delta_0 = \delta_0(m) \in (0,1/C_1)$
so that any stable curve of length at most $\delta_0$ can be cut by
$\cS_{-\ell}$ into at most $K\ell+1$ connected components for all $0 \le \ell \le 2m$.

 Let $\hW^s$
 be, as in \cite[\S 5]{BD1}, the set of (cone-stable) curves whose tangent
vectors  lie in the stable cone for $T$, with length at most $\delta_0$ and curvature
bounded above by a constant $C_{\cK}$ depending only on the table  (homogeneity layers are not used).
The constant $C_{\cK}$ is chosen large enough  that  $T^{-1} \hW^s \subset \hW^s$, up to subdivision of curves.
For $n\ge 1$, $\delta \in (0,\delta_0]$, and $W\in \hW^s$, let  $\cG_{n}^{ \delta}(W)$,
$L_{n}^{ \delta}(W)$, $S_{n}^{ \delta}(W)$, 
and $\cI_n^\delta(W)$ be as in \cite[\S 5]{BD1}:  Set $\cG^\delta_0(W) = W$ and 
define $\cG_n^\delta(W)$ for $n\ge 1$ to be the set of 
smooth components of $T^{-1}W'$ for $W' \in \cG^\delta_{n-1}(W)$, with  elements longer than $\delta$ subdivided to have length between
$\delta/2$ and $\delta$.  More precisely, if a smooth component $U$
has length  $\ell \delta + \rho$ with
$\ell \ge 1$ and $0 \le \rho<\delta$, we decompose $U$ into:\begin{itemize}
\item
either $\ell\ge 2$ pieces
of length $\delta$, if $\rho=0$, 
\item or $\ell\ge 1$ piece(s)
of length $\delta$ and one piece of length $\rho$, placed at one of the
edges of $U$, if $\rho\ge\delta/2$, 
\item
or $\ell-1\ge 0$ piece(s) of length $\delta$, one piece of length
$\delta/2$ (at one tip) and one piece of length $\rho+\delta/2$ (at the other
tip), if $\rho\in (0, \delta/2)$.
\end{itemize} 

 Let $L_n^{\delta}(W)$ denote the set of curves in 
$\cG_n^{\delta}(W)$ that have
length at least $\delta/3$ and
 let $S_n^\delta(W) = \cG_n^\delta(W) \setminus L_n^\delta(W)$. 
For $0\le k<n$, we say that $U \in \cG_k^\delta(W)$ is an ancestor of
$V \in \cG_n^\delta(W)$ if $T^{n-k}V \subseteq U$, and we define $\cI_n^\delta(W)$ to be those curves in 
$\cG_n^\delta(W)$ that have no 
ancestors
of length at least $\delta/3$ (aside from perhaps $W$ itself).  

Finally, let $\delta_1< \delta_0$ and $n_1 \ge m$ be chosen so that \cite[eq. (5.6)]{BD1} holds:
For any stable curve $W$ with $|W| \ge \delta_1/3$ and $n \ge n_1$,
\begin{equation}\label{defn1}
\# L_n^{\delta_1}(W) \ge \tfrac 23 \# \cG_n^{\delta_1}(W) \, .
\end{equation}

Up to replacing $\delta_1$ by a smaller constant, we may and shall only consider
values of $\delta$  of the form 
\begin{equation}\label{convent}
\delta=\delta_0/2^N\, , \quad N\ge 0\, .
\end{equation}
The convention \eqref{convent} is used (only) to allow us to 
ensure that\footnote{
An alternative way to guarantee \eqref{careful} for a fixed length scale $\delta'$ is to
define $\cG_n^{\delta'}(W)$ as usual and treat it as the canonical partition of $T^{-n}W$.  Then for any $\delta'' < \delta'/2$
one can define $\cG_n^{\delta''}(W)$ as a refinement of $\cG_n^{\delta'}(W)$, guaranteeing \eqref{careful}.
This is done implicitly in the proof
of \cite[Lemma~3.11]{BD2} and could be applied in our Lemma~\ref{lem:short small t>} below by taking $\delta' = \delta_{t_*}$
of that lemma.  We do not adopt this approach since  the canonical scale would not be chosen until nearly the end
of our proof.}
for  all $W\in \hW^s$,
\begin{equation}\label{careful}
\forall n\ge 1\, , \mbox{ if } \delta''
<\delta'
\mbox{ then } \forall \, U'' \in L^{\delta''}_n(W) \, , \,  \exists! \, U'\in \cG^{\delta'}_n(W) \mbox{ with } U''\subset U'\, .
\end{equation}
(To prove \eqref{careful} using \eqref{convent}, use induction on $N$, selecting the short tips in a compatible way when dividing $\delta$ by two.)
Property \eqref{careful} is used only
in the proof of Lemma~\ref{lem:short small t>} below.

For $t\ge 0$, we introduce the following shorthand notation,  
$$
S_{n}^{ \delta}(W,t):=
\sum_{W_i \in S_{n}^{ \delta}(W)} |e^{-t \Sigma_{n} \tau}|_{C^0(W_i)}
\, ,\,\,
\cG_{n}^{ \delta}(W,t):=
\sum_{W_i \in \cG_{n}^{ \delta}(W)} |e^{-t\Sigma_{n} \tau }|_{C^0(W_i)}\, ,
$$
and
$$
L_{n}^{ \delta}(W,t):=\cG_{n}^{ \delta}(W,t)-S_{n}^{ \delta}(W,t)\, , \, \, \cI^{\delta}_n(W,t)
:=\sum_{W_i \in \cI_{n}^{ \delta}(W)} |e^{-t\Sigma_{n} \tau }|_{C^0(W_i)}\, .
$$

The lemma below replaces   the
usual one-step expansion (see \cite[Lemma 3.1]{BD2}):

\begin{lemma}[$n$-Step Expansion]
\label{lem:N step}
For any $t_0>0$ and  $\theta_0 \in (e^{-  \tau_{\min}},e^{-  \tau_{\min}/2})$
there exist a finite $n_0(t_0, \theta_0) \ge  2$ and $ \bar \delta_0=\frac{\delta_0}{2^N}>0$ such that 
\begin{equation}
	\label{eq:theta}
	 S_{n_0}^{\bar\delta_0}(W,t) \le  \cG_{n_0}^{\delta_0}(W,t)  < \theta_0^{n_0 t}\, ,\; \;  \forall W \in \hW^s
	\mbox{ with } |W|\le \bar \delta_0  \, ,\, \,
	\forall t\ge  t_0 \, .
\end{equation}
\end{lemma}

See also   \cite[Lemma 3.1(a)]{Ca}.

\begin{proof} Clearly,
$\sup -t\tau \le -t \tau_{\min}<0$ if $t>0$.  For any $n_0\ge 1$, there exists $\bar \delta_0(n_0)=\frac{\delta_0}{2^N}$ such that any  $W\in \hW^s$  with  $|W|<\bar \delta_0$  is such that $T^{-n_0}(W)$
has at most $(Kn_0+1)$ connected components \cite[Lemma 5.2]{Ch}. 
In addition using \cite[Ex. 4.50]{CM} as in \cite[Proof of Lemma 5.1]{BD1},  we have
$|T^{-j}W|\le  C'  |W|^{2^{-s_0 j}}$ for  a uniform $C'>0$ and  all $j\ge 1$  (see also \cite[Lemma~3.1]{Ca}).
Up to taking smaller $\bar \delta_0$, depending on $\delta_0$
(and $n_0$),
we can assume that $|T^{-j}W|\le \delta_0$ for all $0\le j\le n_0$. 
Then, for  $|W|\le \bar \delta_0$,  there can be no additional subdivisions of $T^{-n_0}(W)$ due to pieces growing
longer than $\delta_0$, so that 
\begin{equation}
\label{complexity step}
\cG_{n_0}^{ \delta_0}(W,t) 
\le (K n_0+1) e^{-tn_0  \tau_{\min}} \, .
\end{equation}
The same bound applies to $S_{n_0}^{\bar \delta_0}(W,t)$, since any element of $S_{n_0}^{\bar \delta_0}(W)$
must be created by a genuine cut by a singularity, not an additional subdivision due to pieces growing longer than
$\bar \delta_0$.
For any fixed $t_0>0$ and $\theta_0\in (e^{-  \tau_{\min}},e^{-  \tau_{\min}/2})$, we can find  $n_0=n_0(t_0, \theta_0)\ge  2$ such that 
$(Kn_0+1)^{1/n_0} \le \theta_0^{t_0} e^{\tau_{\min} t_0}$. 
Since $ \theta_0^{t_0} e^{\tau_{\min} t_0}\le  \theta_0^{t} e^{\tau_{\min} t}$ for all $t\ge t_0$, it
follows that \eqref{eq:theta} holds for $\bar \delta_0=\bar \delta_0(n_0,\delta_0)$.
\end{proof}

Lemma~\ref{lem:N step}
 implies  the following analogue\footnote{See  \cite[Lemma 3.1(b)]{Ca} for the replacement for \cite[Lemmas 3.3--3.4, $\zeta\ne 0$]{BD2}, using a logarithmic
weight with $\gamma>0$ as in \cite{BD1}.} of \cite[Lemmas 3.3--3.4, $\zeta=0$]{BD2}:
 
\begin{lemma}[Growth Lemma]
\label{lem:short growth} 
Fix $\theta_0\in (e^{-  \tau_{\min}}, e^{-\tau_{\min}/2})$
and $t_0 > 0$.
Suppose $\delta \le \delta_0$ and $m_1(\delta) \ge n_0(t_0, \theta_0)$ are such that any $W \in \hW^s$
with $|W| \le \delta$ 
has the property that $W \setminus \cS_{-j}$ comprises
at most $Kj + 1$ connected components 
for all $1 \le j \le 2m_1$.
Then for any $t \ge t_0$ and each $W \in \hW^s$ with $|W| \le \delta$, we have
 \begin{equation}
 \label{eq:short growth up}
  \cI^{\delta}_n(W,t)
 	\le \theta_0^{nt} \, , \, \forall n\ge m_1\, , 
 \end{equation}
\begin{equation}
 \label{eq:short growth up small n}
  \cI^{\delta}_n(W,t)
 	\le Km_1 \theta_0^{nt} \, , \,  \forall n < m_1\, , 
 \end{equation}
and, setting $L_0 = \pi \sqrt{1 + \cK_{\min}^{-2}}$,
\begin{equation}
  \label{eq:extra growth up} 
  \cG^{\delta}_n(W,t) \le   \frac{2 L_0  }{C_1 \delta}\, Q_n(t)  \, , \forall  n\ge 1 \, .
\end{equation}
\end{lemma}

 \begin{proof}
 Let $n_0(t_0,\theta_0)$ and $\bar\delta_0(n_0,\delta_0)$ be given by Lemma~\ref{lem:N step}.
 By choice of $n_0$, if $\ve = \tau_{\min} + \log \theta_0 > 0$, then 
 $(Kn_0+1)^{1/n_0} \le e^{\ve t_0}$.  Remark that $(Kn+1)^{1/n}$ decreases to 1 for $n \ge 2$ since $K \ge1$.
 Thus $(Kn+1)^{1/n} \le e^{\ve t_0}$ for all $n \ge n_0$.
 With this observation,  for $\delta$ and $m_1$ as in the statement of the lemma, 
 if $n < m_1$ then \eqref{eq:short growth up small n} follows immediately since each element
 of $\cI^{\delta}_n(W)$ must terminate on an element of $\cS_{n}$,
 \[
 \cI^{\delta}_n(W,t) \le (Kn+1) e^{-tn \tau_{\min}} \le K m_1 \theta_0^{nt} \, .
 \]
 On the other hand, for $n \ge m_1$, we write $n = qm_1 + \ell$, with $q \ge 1$ and $0 \le \ell < m_1$.
 Then, since elements of $\cI_n^\delta(W)$ have been short at each intermediate step, we use \eqref{complexity step} once with $m_1 + \ell$ in place of $n_0$ and $q-1$ times
 with $m_1$ in place of $n_0$ to obtain,
 \[
 \begin{split}
 \cI^\delta_n(W, t) & \le \sum_{V_j \in \cI^\delta_{(q-1)m_1}(W)} \left| e^{-t \Sigma_{(q-1)m_1} \tau} \right|_{C^0(V_j)} \sum_{W_i \in \cI^\delta_{m_1+\ell}(V_j)} \left| e^{-t \Sigma_{m_1+\ell} \tau } \right|_{C^0(W_i)} \\
 & \le (Km_1 +1)^{q-1}(K(m_1+\ell) + 1) e^{-t n \tau_{\min}} \le e^{\ve  t_0 n - t n \tau_{\min}} \, ,
 \end{split}
 \]
 which implies \eqref{eq:short growth up} by choice of $n_0$ and $\ve$.

 Finally, to 
 show \eqref{eq:extra growth up}, first note that each $W_i \in \cG^\delta_n(W)$ is contained in a single element of $\MM_0^{n}$,
  and that multiple $W_i \in \cG^\delta_n(W)$ only belong to the same element of $\MM_0^n$
 as a result of artificial subdivisions at time $n$ or at a previous step.
 Since 
 $|T^{-n}V| \ge C_1 \Lambda^n |V|$ for any stable curve 
 $|V|$ (due to \eqref{eq:hyp}), 
  each such curve must have length at least $C_1 \delta/2$.
 Thus there
can be at most  $2L_0/(C_1\delta)$  elements of $\cG^\delta_n(W)$ in one element of $\MM_0^{n}$,
where $L_0 = \pi \sqrt{1 + \cK_{\min}^{-2}}$ is the maximum length of a stable curve in $\MM_0^n$
using \cite[\S 3]{BD1}. 
Note also that $|e^{-t \Sigma_n \tau}|_{C^0(W_i)} \le |e^{-t \Sigma_n \tau}|_{C^0(A)}$ whenever $W_i \subset A \in \MM_0^n$. This gives the required bound.
\end{proof}


\section{Preparations}\label{boot}

\subsection {Small Singular Pressure. Two Bounds from \cite{Ca}}\label{SSP}

Recall that $n_1$ and $\delta_1$ were defined by \eqref{defn1}.
We say that Small Singular Pressure \#1 (SSP.1)   holds at\footnote{Our formulation of (SSP) corresponds to the choice $\ve=1/4$ in
the formulation of (SSP) in \cite{Ca}, and in the analogous condition appearing in \cite[Cor. 5.3]{BD1}.}
$t\ge 0$ 
if  
		\begin{align}\label{keyb} 
	\mbox{there exist } &\, \delta_t=\delta
	=\frac{\delta_0}{2^{N_t}}\in (0,  \delta_1]\,\,  \mbox{ and a finite }  \, n_t=n_t
	\ge n_1
	\\ 	
	\nonumber \mbox{ such that } & \quad \frac{ S_n^{\delta_t}(W,t)}
	{\cG_n^{\delta_t}(W,t)} \le 
	1/4\, , \,
	\forall n\ge n_t\, , \,  \forall W \in \hW^s  \mbox{ with }   |W| \ge  \delta_t/3\, ,
	\end{align} 
and, in addition,
\begin{align}\label{keyb3} 
 \sum_{n\ge n_t} \,\, \sup_{\substack{W \in \hW^s\, \\    |W|\ge \delta_t/3}}\,\,
\frac{e^{-nt \tau_{\min}} }  { L^{\delta_t}_n(W,t)} < \infty \, 
\end{align} 
together with its
``time-reversal,'' obtained by replacing $T$ with its inverse $T^{-1}$, 
 $\hW^s$ by $\hW^u$,
and replacing $\tau$ with $\tau \circ T^{-1}$ (that is, replacing
$\Sigma_n\tau$ with  $\sum_{i=1}^n \tau \circ T^{-i}= (\Sigma_n \tau )\circ T^{-n}$), both hold.

Assume that \eqref{keyb}  and \eqref{keyb3} hold at
$t\ge 0$ for  $\delta_t$,  and $n_t$. Then we say that  Small Singular Pressure \#2 (SSP.2) holds at $t$
  if\footnote{In the analogous condition of \cite[Cor 5.3]{BD1}, there
exists a uniform $C_t$  such that
$n^*_t(|W|,\delta_t,
1/4)=C_t n_t\frac{|\log (|W|/\delta_{{t}})|}{|\log 1/4|}$.}
		\begin{align}\label{keyb2}
	\mbox{for any } & W \in \hW^s \mbox{ there exists } n^*_t(|W|,\delta_t) \in [n_t, \infty)
	\mbox { such that }\\
	 \nonumber
	 & 
	  \frac{ S_n^{\delta_t}(W,t)}
	{\cG_n^{\delta_t}(W,t)} \le \frac 1 2
	\, , \,   \forall n\ge n^*_t (|W|,\delta_t)\, ,
	\end{align} 
together with its time-reversal (in the sense defined above) both hold.

\medskip

Note that the time-reversal of conditions \eqref{keyb},
\eqref{keyb3}, and \eqref{keyb2} involve stable curves for $T^{-1}$, that is, unstable curves for
$T$.  In view of  the time reversibility of the billiard dynamics (see \cite[Sect.~2.14]{CM} for the precise involution $\iota$), since $\tau \circ T^{-1}=\tau\circ \iota$, and $\tau \circ \iota$ is precisely the free flight time under $T^{-1}$,the conditions for $T$ and $\tau$ are equivalent\footnote{This equivalence does not always hold in \cite{Ca} where $t\tau$
is replaced by a more general  $g$.} with those for $T^{-1}=\iota T\iota$ and $\tau \circ T^{-1}=\tau\circ \iota$.

\medskip

To establish Theorem~\ref{Car}, Carrand proved\footnote{In particular, Carrand shows that
\eqref{keyb}  and \eqref{keyb3} imply the analogues  \cite[Prop. 3.7 and  3.10]{Ca}
	of \cite[Prop. 3.14 and 3.15]{BD2} for the Banach norm of \cite{BD1}. He does not
	get a spectral gap.} the following consequence of SSP:

\begin{proposition}[{\cite[Theorem 1.2]{Ca}}]\label{Car3} Assume\footnote{See also Lemma~\ref{mild}.} \eqref{sparse} and that
 SSP.1  and SSP.2 hold\footnote{SSP.1 suffices to construct the invariant
 measure $\mu_t$ and check it is $T$-adapted. SSP.2 is used to show ergodicity, which gives that
 $\mu_t$ is an equilibrium state for  $-t\tau$, as well as the other claims.} at
$t>0$.
	Then  
	there is  a unique  equilibrium measure $\mu_t$ for  $-t\tau$, this measure is  $T$-adapted,
	 charges nonempty open sets, and is Bernoulli. In addition, $P_*(t)=P(t)$.
\end{proposition}

We state more facts from \cite{Ca} and their consequences. Setting 
\begin{equation}\label{deftC}
t_C=\frac {\log \Lambda}{\tau_{\max} -\tau_{\min}}
>0\, ,
\end{equation}
\cite[Lemmas 3.3  and 3.4 and Corollary 3.6]{Ca} give that 
 each $t\in [0,t_C]$
satisfies SSP (that is, \eqref{keyb}, \eqref{keyb3}, and \eqref{keyb2}) for $\delta_t>0$, $n_t<\infty$,
and $C_t<\infty$. 

\medskip

The key to our 
bootstrap argument is the following definition.

\begin{definition}[Largest SSP Parameter]
\begin{equation}\label{bart}
t_\infty:=\sup\bigl \{ t'\ge 0 \mbox{ such that } \eqref{keyb}, \eqref{keyb3}, \mbox{ and } \eqref{keyb2} \mbox{ hold for all } 0\le t\le t'\bigr \} \, .
\end{equation}
\end{definition}

We will use that $\delta_t$ and $n_t$ exist for all $t< t_\infty$.

By the results of \cite{Ca} recalled after \eqref{deftC}, we already know that $t_\infty\ge t_C>0$. 
We will bootstrap from this fact:
If $P(t_\infty)<0$,  then $t_\infty >\htop(\Phi^1)$, and Proposition~\ref{Car3}
implies Theorem~\ref{main}.
Otherwise,   Lemma~ \ref{lem:short small t>} below will establish  that any $0\le t<s_*$
satisfies \eqref{keyb}, \eqref{keyb3}, and  \eqref{keyb2}, where
 $s_*>t_\infty$ will be  constructed
in  Lemma~\ref{lem:t_*}.

\bigskip

We conclude this section with two key bounds due to Carrand
and a lemma which follows from them.
Assume that \eqref{keyb}  \eqref{keyb3} hold for $t$, then by \cite[Prop 3.7]{Ca} there exists $c_{0, t}>0$  such that
\begin{equation}
\label{eq:lower fixed}
\cG_n^{\delta_t}(W,t) \ge c_{0,t} e^{n P_*(t)}\, ,  \,\,  \forall n \ge 1\, ,\, \forall W\in \hW^s \mbox{ with  } |W| \ge \delta_t/3\, ,
\end{equation}
and by \cite[Prop 3.10]{Ca} there exists $c_{1,t}>0$ such that  
\begin{equation}
\label{eq:upper fixed}
Q_n(t) \le \frac 2 {c_{1,t}} e^{n P_*(t)}\, ,  \,\,  \forall n \ge 1 \, ,
\end{equation}

Observe that \eqref{eq:upper fixed} together with \eqref{eq:extra growth up}  give the upper bound
(to be used in the proof of Lemma~\ref{lem:low t>})
\begin{equation}
\label{eq:upper}
\cG_n^\delta(W,t) \le  { \frac{ 2 L_0}{C_1 \delta} } Q_n(t) 
\le { \frac{ 4 L_0}{C_1 \delta c_{1,t}} } e^{n P_*(t)}
\, ,\,  \forall n \ge 1\,  , \, \forall \delta \le \delta_0 \, .
\end{equation}

\smallskip

Finally, \eqref{keyb} and \eqref{eq:lower fixed} imply the following lower bound for any scale
$\delta = \delta_0/2^N$.

\begin{lemma}\label{thelabel}
For all $t \in (0, t_\infty)$ and  $\delta =\delta_0/2^N$, there exists $c_{0,t}(\delta)>0$ such that 
\begin{equation}
\label{eq:lower fixed'}
\cG_n^{ \delta}(W,t) \ge c_{0,t}(\delta) e^{n P_*(t)}\, ,  \,\,  \forall n \ge 1\, ,\, \forall W\in \hW^s \mbox{ with  } |W| \ge \delta/3\, .
\end{equation}
The time reversal of the statement holds for $T^{-1}$.
\end{lemma}

\begin{proof}
First, assume $\delta < \delta_t$.  Each element of $L_n^{\delta_t}(W)$
contains at least $\delta_t/(3 \delta)$ elements of $\cG_n^\delta(W)$.  So if $|W| \ge \delta_t/3$, then
\eqref{keyb} and bounded
distortion for $\tau$ give
\begin{equation}
\label{eq:first lower}
\cG_n^\delta(W,t) \ge  \frac{e^{- tC}\delta_t}{3 \delta} L_n^{\delta_t}(W,t)
\ge  \frac{e^{- tC}\delta_t}{4 \delta} \cG_n^{\delta_t}(W,t) 
\ge \frac{e^{- tC} \delta_t c_{0,t}}{4 \delta}  e^{n P_*(t)} \,  , 
\end{equation}
for all $n\ge n_t$, where we have used \eqref{eq:lower fixed} in the last step.

Next, if $|W| \in [\delta/3, \delta_t/3)$, then there exists $n_W \le C' \log(\delta_t/\delta)$ such that
$T^{-n_W}(W)$ has a connected component $V$ of length at least $\delta_t/3$.  This is because
while $T^{-n}W$ remains short, the number of components of $T^{-n}W$ is at most $Kn+1$ by
\eqref{eq:complex}
while $|T^{-n}W| \ge C_1 \Lambda^n |W|$ according to \eqref{eq:hyp}.  Thus setting
$\bar n = \max \{ n_W, n_t \}$, we apply \eqref{eq:first lower} to $V$ to estimate for $n \ge \bar n$.
\[
\cG_n^\delta(W,t) \ge \cG^\delta_{n - \bar n}(V, t) e^{- \bar n \tau_{\max}} 
\ge e^{- \bar n (\tau_{\max} + P_*(t) )} e^{- tC} \frac{\delta_t}{4 \delta} c_{0,t} e^{n P_*(t)} \, ,
\]
which proves \eqref{eq:lower fixed'} by definition of $\bar n$.
If $n < \bar n$, then trivially
\[
\cG_n^\delta(W,t) \ge e^{- n \tau_{\max}} 
\ge e^{- n |\tau_{\max} + P_*(t)|} e^{n P_*(t)}
\ge e^{- \bar n |\tau_{\max} + P_*(t)|} e^{n P_*(t)} \, .
\]

Finally, if $\delta \ge \delta_t$, then 
since each element of 
$\cG^{\delta}_n(W)$ contains at most $3\delta/\delta_t$ elements of $L^{\delta_t}_n(W)$ and
$S^{\delta_t}_n(W) \subset S^\delta_n(W)$, we have
\[
\cG_n^{\delta_t}(W,t) = S^{\delta_t}_n(W,t) + L^{\delta_t}_n(W,t)
\le S^{\delta}_n(W,t) + \frac{3 \delta}{\delta_t} \cG^{\delta}_n(W,t) 
\le \Big(1 + \frac{3 \delta}{\delta_t} \Big) \cG^{\delta}_n(W,t) \, ,
\]
which gives the required lower bound on $\cG^\delta_n(W,t)$, applying \eqref{eq:lower fixed}.

The time reversed statement of the lemma follows immediately using the reversibility of the billiard,
as explained earlier. 
\end{proof}


\subsection{First Key Lemma}\label{KL}

We start with the following easy observation:

\begin{lemma}\label{derP} For all $t>0$,
the following limit exists and belongs to   $ [-\tau_{\max},-\tau_{\min}]$:
$$  P'_{*,-}(t):= \lim_{s \uparrow t} \frac{P_*(t)-P_*(s)}{t-s}\, .
$$ 
\end{lemma}

\begin{proof}Existence of the limit follows from the convexity of $P_*(t)$
 which implies that left (and right) derivatives exist at every $t>0$.
 Next, if $0<s<t$, we have 
\begin{equation}
\label{forconv}
  \sum_{A \in \MM_0^n} |e^{-t\Sigma_n \tau}|_{C^0(A)}\le 
 |e^{n(s-t)\tau_{\min}}|
  \sum_{A \in \MM_0^n} |e^{-s\Sigma_n \tau}|_{C^0(A)}\, ,\,\,
  \forall n \ge 1\, ,
\end{equation}
which implies $P'_{*,-}(t)\le -\tau_{\min}$. A similar
computation gives $P'_{*,-}(t)\ge -\tau_{\max}$.
\end{proof}

Our first key lemma in view  of  Lemma~\ref{lem:short small t>} below
is the following adaptation of \cite[Lemma 3.10]{BD2}:

\begin{lemma}
[Using the H\"older Inequality]
\label{lem:low t>}
For all
$0<u \le t< t_\infty$  
and  $\kappa>0$
there exists $\omega_\kappa =\omega_\kappa(u,t)>0$ such that for all $W \in \hW^s$  with
 $|W| \ge {\delta_{u}}/3$,  
\begin{align}
\label{eq:low t> up}
 \cG_n^\delta(W,t)
\ge   \frac {\omega_\kappa(u,t)} \delta  \cdot &e^{n(P_*(u)-(| P'_{*,-}(u)| + \kappa)(t-u) )} \, ,\, \, \\
\nonumber &\forall \delta=\frac{\delta_0 }{2^{N}} \le \delta_{u}
\, , \,\, \forall n\ge n_{u} 
\, .
\end{align}
In addition, for each $\delta =\frac{\delta_0}{ 2^N}<\delta_0$ there exists $\omega_\kappa^* =\omega_\kappa^*(u,t,\delta)>0$ such that for all $W \in \hW^s$  with  $|W| \ge \delta/3$, 
\begin{equation}
\label{eq:low t> up'}
\cG_n^\delta(W,t)
	\ge    \omega_\kappa^*(u,t, \delta) \cdot e^{n(P_*(u)-(| P'_{*,-}(u)| + \kappa)(t-u) )} \, ,\,\, \forall n\ge 1 \, .
\end{equation}
Finally,  the time reversals of  \eqref{eq:low t> up} and \eqref{eq:low t> up'} also hold for the billiard map $T^{-1}$.
\end{lemma}

The proof gives constants $\omega_\kappa(u,t)$ and $\omega_\kappa^*(u,t,\delta)$ which tend to zero as $t\to \infty$
(because the constant $\eta$ in the proof tends to zero as $t\to \infty$).

\begin{proof}
We start with \eqref{eq:low t> up} (for $t \ge u$).
Recall from the proof of \eqref{eq:first lower} that for $u \in (0, t_\infty)$ and $\delta < \delta_u$,
if $|W| \ge \delta_u/3$ and $n\ge n_u$, then
\begin{equation}
	\label{eq:lower}
	\cG_n^\delta(W,u) \ge  e^{-u C}  \frac{\delta_u}{4 \delta} c_{0,u} e^{n P_*(u)} \, ,\,   \forall \delta < \delta_u \,  ,
\end{equation}
since each  $V_i \in L_n^{\delta_u}(W)$
contains at least ${\delta_u}/{3\delta}$ elements of $\cG_n^\delta(W)$.

Now, for $s \in (0,u)$,   taking $\eta(s,t,u)\in (0,1]$ such that 
$\eta t + (1-\eta)s = u$, the H\"older inequality gives
$
\sum_i a_i^{u} 
\le \bigl(\sum_i a_i^t \bigr)^\eta \bigl( \sum_i a_i^{s} \bigr)^{1-\eta} 
$ for any positive numbers $a_i$.
It follows that  for  all $\delta {\le \delta_{u} }$, each $W \in \hW^s$ with $|W| \ge \delta_{u}/3$ and any
$n \ge n_{u}$, 
\begin{align}
\nonumber
 \cG_n^\delta(W,t)
&\ge \frac{( 
 \cG_n^\delta(W,u) )^{1/\eta}}
{(  \cG_n^\delta(W,s))^{(1-\eta)/\eta}}\\
\nonumber &\ge \left(   e^{-u C}    \frac{ \delta_{u}}{4 \delta} c_{0,u} e^{n P_*(u)} \right)^{1/\eta} 
\left( { \frac{  4L_0  }{C_1 \delta c_{1,s}} } e^{n P_*(s)} \right)^{1 - 1/\eta} \\
& = \frac{1}{\delta}
 \left(  e^{-u C}  \frac{\delta_{u}}{4} c_{0,u} \right)^{1/\eta} 
\left ({ \frac{4 L_0  }{ C_1 c_{1,s}}  } \right )^{1-1/\eta} 
e^{n (P_*(u) - P_*(s) )\frac{1-\eta}{\eta} } e^{n P_*(u)} \,   ,
\end{align}
where we used  \eqref{eq:lower} with $u$
for the lower bound in the numerator,
and  \eqref{eq:upper} for $s$ for the upper bound in the denominator, recalling that $\{s,u\} \subset (0, t_\infty)$
and $\delta_{u} \le \delta_1 < \delta_0$. 

Since 
$\eta (s,t,u)= ({u-s})/({t-s})$,
we have 
\begin{equation*}
(P_*(u) - P_*(s))\frac {1-\eta}{\eta}=\frac {t-u}{u-s} (P_*(u)-P_*(s)) \, .
\end{equation*}

 Fix $\kappa>0$ and   choose $s=s(\kappa,u)\in (0, 1)$ close enough to $u$ (i.e. small enough $\eta_\kappa = \eta(s(\kappa, u),t,u) >0$) 
such that (since $0<s<u$ and $ P'_{*,-}(u)<0$ for all $u>0$)
\begin{equation}\label{forkappa}(P_*(s)-P_*(u))/(u-s) \le | P'_{*,-}(u)| + \kappa\, .
\end{equation}
The bound \eqref{eq:low t> up} follows, setting,
for $s=s(\kappa,u)$ (recall that $\eta_\kappa$ depends on $t$),
$$\omega_\kappa(u, t)=\left(  e^{- u C} \frac{\delta_{u}}{4} c_{0,u} \right)^{1/{\eta_\kappa}} 
\left ( { \frac{ 4 L_0  }{C_1 c_{1,s}}  } \right )^{1-1/\eta_\kappa} \, . $$

For \eqref{eq:low t> up'}, we use that  \eqref{eq:upper} for $s$ 
and Lemma~\ref{thelabel} for $u$ imply that
	for any $\delta\in (0, \delta_{u})$, 
	for each $W \in \hW^s$ with  $|W| \ge \delta/3$,
	and all $n \ge 1$,   
	\begin{equation}
	\label{eq:interpol'}
	\cG_n^\delta(W,t)
	\ge \frac{( 
		\cG_n^\delta(W,u) )^{1/\eta}}
	{(  \cG_n^\delta(W,s))^{(1-\eta)/\eta}}
	\ge \bigl (c_{0,u}(\delta) \cdot e^{nP_*(u) }\bigr )^{1/\eta} 
	\bigl(  {  \frac{ 4 L_0  }{C_1 \delta c_{1,s}} } e^{nP_*(s)} \bigr)^{(\eta-1)/\eta}  ,
	\end{equation}
where we used  \eqref{eq:lower fixed'} for $u$.
We conclude
by taking $s=s(\kappa,u) \in (0, 1)$ close enough to $u$ 
such that \eqref{forkappa} holds, setting (again, $\eta_\kappa$ depends on $t$)
\begin{equation*}
\omega_\kappa^*(u, t, \delta)=c_{0,u}(\delta)^{1/\eta_\kappa} ( {  4 L_0 } )^{1-1/\eta_\kappa}( C_1  \delta c_{1,s})^{1/\eta_\kappa-1}\, .
\end{equation*}
\end{proof}

\section{Proof of Theorem~\ref{main}}
\label{44}
\subsection{Choosing the Pivot $t_*$.\label{pivot}}

The next lemma is (inspired by \cite[Definition 3.9]{BD2}). Recall $-\tau_{\max}\le P'_{*,-}(t)\le -\tau_{\min}$ from
Lemma~\ref{derP}.

\begin{lemma}[Pressure Gap: Constructing the  ``{Pivot}'' $t_*$]
\label{lem:t_*}
For any $t>0$ and $\theta_0\in (e^{-  \tau_{\min}},e^{-  \tau_{\min}/2})$, 
defining 
\begin{equation} 
\label{s*}
s_*(t) := \frac{t  |P'_{*,-}(t)|}{| P'_{*,-}(t)|+ (\log \theta_0)/2}
\, , \, \, t\in (0, t_\infty)\, ,
\end{equation}
there exists $t_*\in (0,t_\infty)$ such that $s_*:=s_*(t_*)> t_\infty$.
\end{lemma}

\begin{remark}\label{choice}
The parameter $s_*(t_*)=s_*(t_*,\theta_0) > t_*$ is defined such that
\begin{equation}\label{defs*}
\theta_0^{s_*/2} e^{|P_-'(t_*)|(s_* - t_*)} = 1 \, .
\end{equation}
The reason for this will become clear in the proof of Lemma~\ref{lem:short small t>}.
In particular, we shall use the  value of  $\theta_0$  from Lemmas~\ref{lem:N step} and~\ref{lem:short growth}
for $t_*=t_*(\theta_0)$ and $s_*(t_*)=s_*(t_*,\theta_0)$ in Lemma~\ref{lem:short small t>}.
Note also that replacing $(\log \theta_0)/2$ by $a \log \theta_0$
in \eqref{s*} and taking $\theta_0\in (e^{-  \tau_{\min}},e^{-  b\tau_{\min}})$, for $a, b \in (0,1)$,
would replace $1/2$ by $a$ in \eqref{defs*}, \eqref{check1}, \eqref{check2}, 
\eqref{choice'} (and the line above it),  \eqref{nofoot},
and (thrice) in the two lines after \eqref{choicekappa}, and it would replace $4$ by $(ab)^{-1}$ in \eqref{check}, \eqref{abo}, and \eqref{soon}. Taking $a$ and $b$ close to $1$, this would give a larger value for $s_*$ (up to taking $\kappa$ smaller in \eqref{choicekappa}). Since  $e^{n b\tau_{\min}}$
is a rough bound on the $n$-step expansion of Lemma~\ref{lem:N step},
and (more importantly) our argument is by contradiction,  there is no reason to optimise here.
\end{remark}

\begin{proof}
To construct $t_*$, we first check that 
\begin{equation}
\label{check}s_*(t)>t \cdot \bigl (1+ \frac{\tau_{\min}}{4\tau_{\max}} \bigr )\, , \, \, 
\forall t\in (0 , t_\infty)
\, .
\end{equation}
Indeed, since 
\begin{equation}
\label{check1}
 \frac 1{1-\frac{|\log \theta_0|}{2| P'_{*,-}(t)|}}
 > 1+\frac{|\log \theta_0|}{2| P'_{*,-}(t)|}\, ,
\end{equation} 
the bound \eqref{check} follows from the fact that $ \tau_{\min} \le  | P'_{*,-}(t)|\le \tau_{\max}$ implies
\begin{equation}
\label{check2}
\frac{|\log   \theta_0 |}{2| P'_{*,-}(t)|}
\in \Bigl [ \frac{\tau_{\min}}{4\tau_{\max}},\frac  1 {2} \bigr )\, .
\end{equation}
Then, taking $t_*=t_\infty-\upsilon$ for  $\upsilon \in (0,t_\infty)$, 
it suffices to pick $\upsilon>0$  such that
\begin{equation}\label{abo}
\bigl (1+ \frac{\tau_{\min}}{4\tau_{\max}} \bigr )\bigl ( t_\infty-\upsilon \bigr) 
 > t_\infty \, .
\end{equation}
Since $t_\infty \ge t_C=\log \Lambda/ (\tau_{\max}-\tau_{\min})$ by \eqref{deftC}, the bound 
\eqref{abo} holds as soon as
\begin{equation}\label{soon}
\upsilon < \log \Lambda \cdot (\tau_{\max}-\tau_{\min})^{-1}\cdot  
\bigl (1+ 4\frac {\tau_{\max}}{\tau_{\min}}\bigr )^{-1}  \, .
\end{equation}
\end{proof}

\subsection{Second Key Lemma}\label{2ndl}
The second key lemma is inspired by \cite[Lemma 3.11]{BD2}
(the proof below requires a more involved decomposition of orbits):

\begin{lemma}
\label{lem:short small t>} Fix $\theta_0\in (e^{-\tau_{\min}}, e^{-\tau_{\min}/2})$.
Let $t_*<t_\infty$ and $s_*(t_*)>t_\infty$ be as in Lemma~\ref{lem:t_*}. If $P_*(t_*)\ge 0$
then  the SSP conditions \eqref{keyb}, \eqref{keyb3},  and \eqref{keyb2} hold  at all $t\in [ t_*, s_*)$.
\end{lemma}

The proof below uses \eqref{careful} and thus the convention \eqref{convent}.

\begin{proof}[Proof of Lemma~\ref{lem:short small t>}]
We first consider condition \eqref{keyb} of SSP.1.

By definition of $s_*$ (recall that $\inf |P'_{*,-}(s)|>-\log \theta_0/2$)
\begin{equation}\label{choice'}
\theta_0^{t'/2} e^{| P'_{*,-}(t_*)| (t'-t_*)} < 1\, ,\quad
	\forall t_*\le t' <  s_*\, .
\end{equation}
Thus for all $t' \in [t_*, s_*)$ 
there exists $ \kappa_1 = \kappa( t_*,  t'  ) >0$  such that
\begin{equation}
	\label{nofoot} 
\bar \ve := \sup_{ t_*\le t\le  t'}\bigl (	 \theta_0^{t/2}  e^{(| P'_{*,-}(t_*)| + \kappa_1)(t-t_*)}\bigr )  < 1\, . 
\end{equation}
For 
$m_1\ge  \max\{n_0(t_*,\theta_0),n_{t_*}\}$  to be chosen later depending on  
$\bar \ve$, $\delta_{t_*}$,   and $\kappa_1$,  pick $\delta_3(m_1) { \in (0, \delta_{t_*}]}$ (similarly to the choice of $\bar \delta_0$ in the proof of Lemma ~\ref{lem:N step})
so small that any stable curve of length at most $\delta_3$ can be cut into at most
    $Kj+1$ connected components by $\cS_{-j}$ for $0 \le j \le 2m_1$.

For $n \ge m_1$, write $n = \ell m_1 + r$, for some $0 \le r < m_1$ and $\ell\ge 1$.  Let $W \in \hW^s$ with $|W| \ge \delta_3/3$.  
We group the curves
$W_i \in S^{\delta_3}_n (W)$ with $|W_i| < \delta_3/3$, as in the proof of   \cite[Lemma~3.11]{BD2},
according to the largest  $k\in \{0, \ldots, \ell-1\}$ such that $T^{(\ell-k)m_1+ r}W_i \subset V_j \in L^{\delta_3}_{km_1}(W)$
(such a $k$ must exist since $|W|\ge \delta_3/3$ while $|W_i|<\delta_3/3$).  
Denote\footnote{Note that $\bar \cI^{\delta}_{(\ell-k)m_1 + r}(V_j)$
was abusively denoted $\cI^{\delta}_{(\ell-k)m_1 + r}(V_j)$ in the proof of \cite[Lemma~5.2]{BD1}, see footnote 23  there.} 
by $\bar \cI^{\delta_3}_{(\ell-k)m_1 + r}(V_j)$ the set of $W_i \in \cG^{\delta_3}_n(W)$ thus associated with 
$V_j \in L_{km_1}^{\delta_3}(W)$ (such elements are known to be small  only at iterates
 $jm_1+r$). For such $W_i$, 
 $T^{(\ell-k')m_1+r}(W_i)$ is contained in an element of
$\cG_{m_1 k'}^{\delta_3}  (W)$ shorter than $\delta_3/3$ for $k'<k$.  
   So for $k >0$,
we may apply the inductive bound \eqref{eq:short growth up} since elements of 
$\bar \cI^{\delta_3}_{(\ell-k)m_1+r}(V_j)$ can only be created by intersections with $\cS_{-m_1}$ at the first $\ell - k-1$ iterates and with $\cS_{-m_1-r}$ at the last step.  
 For $k=0$, $W$ itself may be longer than $\delta_3$.  Thus we first subdivide $W$ into at most
$\delta_0/\delta_3$ curves of length at most $\delta_3$ and then apply \eqref{eq:short growth up} to
each piece.
This yields, for $t_*\le t\le  t'$,
\begin{align}
\nonumber
S_n^{\delta_3}(W,t)
& \le \sum_{k=0}^{\ell-1} \sum_{V_j \in L^{\delta_3}_{km_1}(W)} |e^{-t\Sigma_{km_1} \tau}|_{C^0(V_j)}
\sum_{W_i \in \bar\cI^{\delta_3}_{(\ell-k)m_1 + r}(V_j)} |e^{-t\Sigma_{(\ell-k)m_1+r} \tau}|_{C^0(W_i)} \\
\label{eq:short first} 
& \le  \frac{\delta_0}{\delta_3} \theta_0^{tn} + 
\sum_{ k=1}^{\ell-1} \sum_{V_j \in L^{\delta_3}_{km_1}(W)} |e^{-t\Sigma_{k m_1} \tau}|_{C^0(V_j)}  
 \theta_0^{t { ((\ell-k)m_1 + r) } } \, .
\end{align}
Next, recalling \eqref{careful}, for any $k \ge 1$,  each $V_j \in L^{\delta_3}_{k m_1}(W)$ is contained in an element
$U_i \in \cG_{k m_1}^{\delta_{t_*}}(W)$.  Since $|V_j| \ge \delta_3/3$, there are at most $3 \delta_{t_*}/\delta_3$ 
different  $V_j$ corresponding to each fixed $U_i$.  Then we group each $U_i \in \cG_{km_1}^{\delta_{t_*}}(W)$ according to
its most recent long ancestor $W_a \in L_j^{\delta_{t_*}}(W)$ for some $j \in [0, k m_1]$.
  Note that
$j=0$ is possible if $|W| \ge \delta_{t_*}/3$. If $|W|  < \delta_{t_*}/3$, and no such time $j$ exists for $U_i$, then by
convention we 
also associate the index $j=0$ to such $U_i$.  In either case, $U_i \in \cI_{km_1}^{\delta_{t_*}}(W)$,
and we may apply \eqref{eq:short growth up} after possibly subdividing $W$ into at most
$\delta_0/\delta_{t_*}$ curves of length at most $\delta_{t_*}$.
Then,  for $j \ge 1$, we apply    
\eqref{eq:short growth up small n} from 
Lemma~\ref{lem:short growth} 
to each  $\cI_{km_1 - j}^{\delta_{t_*}}(\cdot)$
(since $\delta_3 \le \delta_{t_*}$, the constant $m_1(\delta_{t_*}) \le m_1(\delta_3)$, so the bound holds
with our chosen $m_1$, although it may not be optimal),  
\begin{align*}
 L^{\delta_3}_{km_1}(W, t)
& \le \frac{3 \delta_{t_*}}{\delta_3} 
\biggl( \sum_{U_i \in \cI_{km_1}^{\delta_{t_*}}(W) }  |e^{-t\Sigma_{k m_1} \tau}|_{C^0(U_i)}     \\
& \qquad  + 
\sum_{j=1}^{km_1} \sum_{W_a \in L_j^{\delta_{t_*}}(W)}   |e^{-t\Sigma_{j} \tau}|_{C^0(W_a)} \sum_{U_i \in \cI_{km_1 -j}^{\delta_{t_*}}(W_a)}
 |e^{-t\Sigma_{k m_1-j} \tau}|_{C^0(U_i)}   \biggr ) \\
& \le \frac{3 \delta_{t_*}}{\delta_3} \biggl (  \frac{\delta_0}{\delta_{t_*}}  \theta_0^{t km_1 } + \sum_{j=1}^{km_1} \sum_{W_a \in L_j^{\delta_{t_*}}(W)}   |e^{-t\Sigma_{j} \tau}|_{C^0(W_a) } 
 { Km_1 } \theta_0^{t (km_1 - j)} \biggr) \, .
\end{align*}
Combining this estimate with \eqref{eq:short first} yields  (summing over $k$ for the $j=0$ terms and adding the term corresponding to $k=0$),
\begin{equation}
\label{<}
S_n^{\delta_3}(W,t)
\le  \frac{3 \delta_0}{\delta_3} \frac{n}{m_1} \theta_0^{t n} +  
 \frac{3 { \delta_{t_*} } }{\delta_3}   
\sum_{k=1}^{\ell-1}  \sum_{j=1}^{km_1}
 { Km_1 }  \theta_0^{t (n -j)} L_j^{\delta_{t_*}}(W,t)  \, .
\end{equation}

For fixed $k\in \{1, \ldots, \ell-1\}$, and for each $1\le j\le km_1$ such that $L_j^{\delta_{t_*}}(W)\ne \emptyset$, 
the lower bound \eqref{eq:low t> up} in Lemma~\ref{lem:low t>} (for $u=t_*$)
 and the distortion constant $e^{-tC}\ge e^{-t'C}$ imply (note that 
 $n-j\ge \ell m_1+r -km_1\ge r +m_1\ge n_{t_*}$),
\begin{align}
 \nonumber  \cG_n^{\delta_3}(W,t)
& \ge 
\sum_{W_a \in L_j^{\delta_{t_*}}(W)} e^{-tC} 
 |e^{-t\Sigma_{j} \tau}|_{C^0(W_a)} 
 \sum_{W_i \in \cG_{n-j}^{\delta_3}(W_a)}
  |e^{-t\Sigma_{n-j} \tau}|_{C^0(W_i)} 
 \\
\label{>} & \ge \frac{\omega_{\kappa_1}(t_*,t)}{\delta_3 e^{t'C} }   e^{(n-j)(P_*(t_*)-(| P'_{*,-}(t_*)| + \kappa_1)(t-t_*) )} 
\sum_{W_a \in L_j^{\delta_{t_*}}(W)} | e^{-t\Sigma_{j} \tau}|_{C^0(W_a)}  \, .
\end{align}
Combining \eqref{<} with either \eqref{>} (for $j\ge 1$) or \eqref{eq:low t> up'} from Lemma~\ref{lem:low t>}  (for  $j=0$ and $u=t_*$) and setting $\Delta= 3 e^{t'C} \delta_{t_*} Km_1 $,
yields (using that $P_*(t_*)\ge 0$),
\begin{align}
\nonumber  \frac{ S_n^{\delta_3}(W,t)}{ \cG_n^{\delta_3}(W,t) }
& \le n
\frac{ \frac{3  \delta_0 }  {\delta_3 m_1} 
\theta_0^{t n}}{\omega_{\kappa_1}^*(t_*, t, \delta_3)
e^{n(P_*(t_*)-(| P'_{*,-}(t_*)| + \kappa_1)(t-t_*) )}
} \\
\nonumber & \qquad
 + \sum_{k=1}^{\ell-1} \sum_{j=1}^{km_1} 
 \frac{ \frac{3  \delta_{t_*}}{\delta_3}  { Km_1}  \theta_0^{t (n - j)} 
  L_j^{\delta_{t_*}}(W,t)  }
 {\frac{\omega_{\kappa_1}(t_*,t)}{\delta_3 e^{t'C}} e^{(n-j)(P_*(t_*)-(| P'_{*,-}(t_*)| + \kappa_1)(t-t_*) )}
  L_j^{\delta_{t_*}}(W,t)  } \\
\nonumber  & \le
 \frac{3  \delta_0 }  {\delta_3 \cdot \omega_{\kappa_1}^* (t_*,t,\delta_3)\cdot m_1} n (e^{-P_*(t_*)} \bar \ve)^n
 + \frac{ \Delta  }{\omega_{\kappa_1}(t_*,t)}\sum_{k=1}^{\ell-1} \sum_{j=1}^{km_1}  (e^{-P_*(t_*)}\bar \ve)^{n-j} \\
 \nonumber & \le \frac{ 3  \delta_0  }{\delta_3 \cdot \omega_{\kappa_1}^*(t_*,t,\delta_3) \cdot m_1} n \bar \ve^n
 +  \frac{ \Delta }{\omega_{\kappa_1}(t_*,t)}\frac 1 {1-\bar \ve} \sum_{k=1}^{\ell-1} \bar \ve^{n-km_1}
 \\& \label{endSSP1}
 \le \frac{ 3  \delta_0   }{\delta_3 \cdot \omega_{\kappa_1}^*(t_*, t,\delta_3) \cdot m_1}n \bar \ve^n
 + \frac{ 3 e^{t'C} { \delta_{t_*} Km_1 } }{\omega_{\kappa_1}(t_*,t)\cdot } \frac{\bar \ve^{m_1}}{(1-\bar \ve)(1-\bar \ve^{m_1})}\, .
   \end{align}
To establish \eqref{keyb}, choose first $m_1\ge n_{t_*}$  such that
the second term is less than $1/8$ 
setting $\delta_{t}:=\delta_3(m_1)$, and then $n_{t} \ge m_1$   such that the first term is less than
$1/8$, for  $n \ge n_{t}$. 

 We next show \eqref{keyb3}.  
 For $n\ge n_t$, we deduce from \eqref{keyb} and \eqref{eq:low t> up'} (for small $\kappa>0$) that,
  for all $W\in \hW^s$  with $|W|\ge \delta_t /3$,
 \begin{align}\label{choicekappa}
 L_n^{\delta_t}(W,t) &\ge  \frac {3}{4} \cG_n^{\delta_t}(W,t)\ge \frac {3}{4}\omega_\kappa^*(t_*,t,\delta_t)  e^{nP_*(t_*)} e^{-n (t-t_*)( |P'_{*,-}(t_*)|+\kappa)}\, .
 \end{align}
 Since  
$  e^{-| P'_{*,-}(t_*)| (t-t_*)} > \theta_0^{t/2}\ge e^{-t\tau_{\min}/2}$ by \eqref{choice'},
 while  $P_*(t_*)\ge 0$, it suffices to take $\kappa$ such that
 $
 (t-t_*)\kappa +\frac{t}{2}\tau_{\min} < t \tau_{\min}
 $ to complete the proof of \eqref{keyb3}.
 
\medskip

It remains to consider   SSP.2. 
We may assume $|W|<\delta_{t_*}/3$ since otherwise  \eqref{keyb} from SSP.1 implies
\eqref{keyb2} with $n_{t}^*=n_t$. 
As observed in the proof of \cite[Cor. 5.3]{BD1},
there exists $\bar C_2$ (depending only on the billiard
table) 
such that the first iterate $\ell_0$ at which $\cG^{\delta_{t_*}}_{\ell_0}(W)$ contains
at least one element of length more than $\delta_{t_*}/3$ satisfies 
$$\ell_0 \le n_2=n_2(\delta_{t_*}) := \bar C_2 |\log(|W|/\delta_{t_*})|\, .$$

Since $|W|<\delta_{t_*}/3$, it suffices to consider
the term corresponding to $j=0$  (and $k=0$)  in \eqref{endSSP1} (the other one is bounded by
$1/8$, for $n\ge m_1$ for $m_1$ chosen as above).
For this purpose, for any $n=\ell m_1+r \ge m_1$,  the first term of  \eqref{<} is replaced
by
\begin{equation}
\label{<2}
\frac{\delta_{t_*}}{3 \delta_3} \theta_0^{tn} + \sum_{k=1}^{\ell-1} \frac{3\delta_{t_*}}{\delta_3} \theta_0^{t n}
\le  \frac{3 \delta_{t_*} n}{\delta_3 m_1}   \theta_0^{t n }  \, ,
\end{equation}
where we have applied \eqref{eq:short growth up} from Lemma~\ref{lem:short growth}. 
For any $n\ge \max\{n_2, m_1\}$, the bound \eqref{eq:low t> up'} from Lemma~\ref{lem:low t>}
(for $u=t_*$)  is replaced by
\begin{equation}
\label{eq:low t> up'2}
\cG_n^{\delta_3}(W,t)
	\ge    
\omega_{\kappa_1}^*(t_*,t, \delta_3) \cdot e^{-tn_2\tau_{\max}} e^{(n-n_2)(P_*(t_*)-(| P'_{*,-}(t_*)| + \kappa_1)(t-t_*) )}  \, .
\end{equation}
Dividing \eqref{<2} by \eqref{eq:low t> up'2}, the term corresponding to $j=0$ in \eqref{endSSP1} is bounded by
\begin{align*}
&\frac{3 { \delta_{t_*} }  \frac{n}{m_1} \theta_0^{t n } }
{\delta_3 \cdot \omega_{\kappa_1}^*(t_*,t, \delta_3) \cdot e^{-tn_2\tau_{\max}} e^{(n-n_2)(P_*(t_*)-(| P'_{*,-}(t_*)| + \kappa_1)(t-t_*) )}}\\
&\qquad\qquad\le \frac{3 { \delta_{t_*} }     e^{tn_2\tau_{\max}}}
{m_1 \cdot \omega_{\kappa_1}^*(t_*,t, \delta_3)\cdot  \delta_3 }
n \bar \ve^{n-n_2}\, .
\end{align*}
We conclude, since, if $n^*_t/n_2$ is large enough (depending on $t$, $\bar \ve$, $\delta_3=\delta_t$) then 
$$
n(\bar \ve^{n/n_2} e^{t\tau_{\max}})^{n_2}< 
\frac 1 8 \cdot\frac{\bar \ve^{n_2} \cdot { m_1 \cdot } \delta_3\cdot \omega_{\kappa_1}^*(t_*,t, \delta_3) }
{3 { \delta_{t_*} } }\,, \, \, \forall n\ge n^*_t\, .
$$
\end{proof}

\subsection{Proof of Theorem~\ref{main}}\label{proofmain}

If $P(t_\infty)< 0$ then $t_\infty >\htop(\Phi^1)$, using
Proposition~\ref{Car2}, and we are done
by Proposition~\ref{Car3} and the definition of $t_\infty$,
since we assumed \eqref{sparse}  at $\htop(\Phi^1)$.
Assume for a contradiction
that $P(t_\infty)\ge 0$. Let $t_*<t_\infty$ and $s_*(t_*)>t_\infty$  be as in Lemma~\ref{lem:t_*},
and  fix $t_\infty<t_2<s_*$. 
Since $P_*(t_*)> P_*(t_\infty)\ge P(t_\infty)$ (by Theorem~\ref{Car}(a) applied to
$s=t_\infty$ and $t=t_*$), our assumption that $P(t_\infty)\ge 0$ 
implies that $P_*(t_*)>  0$. Then
Lemma~\ref{lem:short small t>} applied to  $t_*$ and $s_*(t_*)$
gives  that the SSP conditions \eqref{keyb}, \eqref{keyb3}, and \eqref{keyb2} hold for all $t \in [0, t_2]$.
Since $t_2>t_\infty$, this is a contradiction, which concludes the proof
of Theorem~\ref{main}.


\end{document}